\documentclass[11pt]{amsart}
\usepackage{}
\usepackage{amsmath}\usepackage{amscd}\usepackage{amssymb}\usepackage{amsfonts}
\newtheorem{theorem}{Theorem}[section]
\newtheorem{lemma}[theorem]{Lemma}
\newtheorem{corollary}[theorem]{Corollary}

\theoremstyle{definition}

\theoremstyle{remark}
\newtheorem{remark}[theorem]{Remark}
\numberwithin{equation}{section}

\begin{document}
\title[Myers' type theorem with the Bakry-\'Emery Ricci tensor]
{Myers' type theorem with the Bakry-\'Emery Ricci tensor}
\author{Jia-Yong Wu}
\address{Department of Mathematics, Shanghai Maritime University, 1550 Haigang Avenue, Shanghai 201306, P. R. China}\email{jywu81@yahoo.com}

\date{\today}

\begin{abstract}
In this paper we prove a new Myers' type diameter estimate on a complete connected
Reimannian manifold which admits a bounded vector field such that the
Bakry-\'Emery Ricci tensor has a positive lower bound. The result is
sharper than previous Myers' type results. The proof uses the generalized mean
curvature comparison applied to the excess function instead of the classical second variation
of geodesics.
\end{abstract}
\maketitle

\section{Introduction}\label{Int1}
Let $(M,g)$ be an $n$-dimensional complete connected Riemannian manifold. For any
smooth vector field $V$ on $M$, the $m$-Bakry-\'Emery Ricci tensor \cite{[BE]} is
defined by
\[
\mathrm{Ric}^m_V:=\mathrm{Ric}+\frac 12\mathcal{L}_V g-\frac{1}{m}{V^*}\otimes{V^*}
\]
for some number $m>0$, where $\mathrm{Ric}$ is the Ricci tensor of $(M,g)$,
$\mathcal{L}_V$ denotes the Lie derivative in the direction of $V$, and $V^*$
is the metric-dual of $V$. When $m=\infty$, one denotes
\[
\mathrm{Ric}_V:=\mathrm{Ric}_V^\infty=\mathrm{Ric}+\frac 12\mathcal{L}_V g.
\]
In particular, if
\[
\mathrm{Ric}_V=\lambda\, g
\]
for some $\lambda\in\mathbb{R}$, then $(M, g)$ is a Ricci soliton, which is
a natural generalizations of Einstein metric and plays a fundamental role in
the formation of singularities of the Ricci flow \cite{[Ham]}. A Ricci soliton
is called expanding, steady or shrinking, if $\lambda<0$, $\lambda=0$, or
$\lambda>0$, respectively. When $V=\nabla f$ for some $f\in C^\infty(M)$,
the Ricci soliton becomes a gradient Ricci soliton:
\[
\mathrm{Ric}+\mathrm{Hess}\, f=\lambda\, g,
\]
where $\mathrm{Hess}$ is the Hessian of the metric $g$. By Perelman \cite{[Pere]},
any compact Ricci soliton is necessarily a gradient soliton; by Hamilton
\cite{[Ham]} and Ivey \cite{[Ivey]}, any compact gradient non-shrinking Ricci
soliton is Einstein. Hence, any compact non-shrinking Ricci soliton must be
Einstein. A compact shrinking Ricci soliton, when the dimension of
manifold is two or three, is also Einstein (see Hamilton \cite{[Ham]} and
Ivey \cite{[Ivey]}); but when $n\geq 4$, there exist nontrivial compact
gradient shrinking solitons (see Cao \cite{[Cao96]}, Koiso \cite{[Koi]},
Wang-Zhu \cite{[WaZh]}). On the other hand, there also exist many nontrivial
examples of noncompact Ricci solitons; see for instance the survey \cite{[Cao]}.

\vspace{.1in}

In \cite{[FG]}, Fern\'andez-L\'opez  and Garc\'ia-R\'io proved a Myers' type diameter
estimate when the Bakry-\'Emery Ricci tensor has a positive lower bound. That is, if
\[
\mathrm{Ric}_V \geq (n-1)H\,g
\]
for some real constant $H>0$, and $|V|\leq a$ for some real constant $a\ge0$, then
$M$ is compact. Later, Limoncu \cite{[Lim]} analyzed the index form of a minimal
unit speed geodesic, and gave a explicit upper bound to the diameter of $M$,
\[
\mathrm{diam}(M)\leq\frac{\pi}{(n-1)H}\left(\frac{a}{\sqrt{2}}+\sqrt{\frac{a^2}{2}+(n-1)^2H}\right).
\]
Furthermore, Tadano (see Theorem 2.1 in \cite{[Tad2]}) improved an inequality following
Limoncu's proof strategy, and sharpened the diameter estimate,
\begin{equation}\label{oldest1}
\mathrm{diam}(M)\leq\frac{1}{(n-1)H}\left(2a+\sqrt{4a^2+(n-1)^2H\pi^2}\right).
\end{equation}

\vspace{.1in}

In this paper, we take a different approach to get a sharper diameter
estimate than \eqref{oldest1}. Instead of using the second variation of a minimal
unit speed geodesic segment, we apply the generalized mean curvature comparison
to the excess function and get a new Myers' type theorem. We remark that the
excess function was also used by Petersen and Sprouse \cite{[PS]} to prove a
Myers' theorem under integral curvature bounds. Our proof is motivated by
the argument on the diameter estimate of smooth metric measure spaces
$(M,g,e^{-f}dv)$ in the case when $f$ is bounded \cite{[WW]}.

\begin{theorem}\label{Thm}
Let $(M,g)$ be an $n$-dimensional complete connected Riemannian manifold which
admits a smooth vector field $V$ satisfying
\[
\mathrm{Ric}_V \geq (n-1)H\,g
\]
for some real constant $H>0$, and $|V|\leq a$ for some real constant $a\ge0$.
Then $M$ is compact and the diameter satisfies
\begin{equation}\label{newest}
\mathrm{diam}(M)\, \leq\frac{2a}{(n-1)H}+\frac{\pi}{\sqrt{H}}.
\end{equation}
\end{theorem}

The bound assumption on $V$ is necessary as explained by Wei and Wylie (see
Examples 2.1 and 2.2 in \cite{[WW]}). It is easy to see that \eqref{newest}
is sharper than \eqref{oldest1} for any constants $H>0$ and $a>0$. When $a=0$,
Theorem \ref{Thm} recovers the classical Myers' compact theorem \cite{[Myer]}.
\begin{remark}
There have been other Myers' type theorems involving the Bakry-\'Emery Ricci
tensor under different conditions; see Bakry and Qian \cite{[BQ]}, Cavalcante,
Oliveira and Santos \cite{[COS]}, Li \cite{[Li]}, Limoncu \cite{[Lim2]}, Lott
\cite{[Lot]}, Mastrolia, Rimoldi and Veronelli \cite{[MRV]}, Morgan \cite{[Mor]},
Soylu \cite{[Soy]}, Tadano \cite{[Tad1]}, Wei and Wylie \cite{[WW]} and
Zhang \cite{[Zhang]} for details.
\end{remark}

For a fixed point $p\in M$, let $r(x):=d(x,p)$ be a distance function from $p$ to $x$.
In geodesic polar coordinates at $p$, let $\nabla r=\partial_r$. When $V=\partial_rf$
for some $f\in C^\infty(M)$, we have a Myers' type result under only a lower bound of
$\partial_rf$.
\begin{theorem}\label{Thm2}
Let $(M,g)$ be an $n$-dimensional complete connected Riemannian manifold which
admits a smooth function $f$ satisfying
\[
\mathrm{Ric}+\mathrm{Hess}\, f\ge (n-1)H\,g
\]
for some real constant $H>0$. If $\partial_r f\ge-a$ for some constant $a\geq 0$,
along a minimal geodesic segment from every point $x\in M$, then $M$ is compact and
the diameter satisfies
\[
\mathrm{diam}(M)\, \leq\frac{2a}{(n-1)H}+\frac{\pi}{\sqrt{H}}.
\]
\end{theorem}

\begin{remark}
The condition ``every point $x\in M$" in Theorem \ref{Thm2} cannot weakened to
``a fixed point $x\in M$". A obvious counterexample is a Gaussian shrinking
Ricci soliton.
\end{remark}

In particular, for compact shrinking gradient Ricci solitons, Theorem
\ref{Thm} implies that
\begin{corollary}\label{cor}
Let $(M,g)$ be an $n$-dimensional compact connected Riemannian manifold which
admits a smooth function $f$ satisfying
\[
\mathrm{Ric}+\mathrm{Hess}\, f=\lambda\, g
\]
for some real constant $\lambda>0$. Then the diameter of $M$ satisfies
\[
\mathrm{diam}(M)\, \leq\frac{2\sqrt{R_{max}-R_{min}}}{\lambda}+\sqrt{\frac{n-1}{\lambda}}\,\,\pi,
\]
where $R_{max}$ and $R_{min}$ denote the maximum and minimum values of the scalar curvature on $(M,g)$,
respectively.
\end{corollary}

Recall that Tadano gave a diameter estimate on compact connected
shrinking Ricci solitons (Theorem 1.2 in \cite{[Tad2]})
\[
\mathrm{diam}(M)\, \leq\frac{2\sqrt{R_{max}-R_{min}}}{\lambda}
+\frac{1}{\lambda}\sqrt{4(R_{max}-R_{min})+(n-1)\lambda\pi^2}.
\]
Al\'{\i}as et al. proved this type inequality in a more general
setting (see Theorem 8.7 in \cite{[AMR]}). Obviously, our estimate is sharper than
their result.

\vspace{.1in}

The rest of this paper is organized as follows. In Section \ref{sec2}, we prove
two mean curvature comparisons for $\mathrm{Ric}_V$. One is general (see Theorem
\ref{BasicMean}), requiring no assumptions on the vector field $V$. The other
has a stronger conclusion (see Theorem \ref{meancomp}), requiring a bound of $V$.
In Section \ref{sec3}, we apply a argument of Wei and Wylie \cite{[WW]} to prove
Theorems \ref{Thm} and \ref{Thm2}. The proof uses Theorems \ref{BasicMean}
and \ref{meancomp} to the excess function. In Section \ref{sec4}, we apply Theorem
\ref{Thm} and properties of compact shrinking Ricci solitons to prove Corollary
\ref{cor}.

\vspace{.1in}

\textbf{Acknowledgement}.
The author would like to thank anonymous referees for pointing out
many expression errors and give many valuable suggestions that helped
to improve the presentation of the paper. This work is supported by the
NSFC (11671141) and the Natural Science Foundation of Shanghai (17ZR1412800).


\section{Mean curvature comparison for $\mathrm{Ric}_V$}\label{sec2}
Let $(M,g)$ be an $n$-dimensional complete connected Riemannian manifold. For a
fixed point $p\in M$, let $r(x):=d(x,p)$ be a distance function from $p$ to $x$.
Then $r(x)$ is smooth for all $x\in M\setminus\{p, C_p\}$, where $C_p$ denotes
the cut locus of the point $p$. In geodesic polar coordinates at $p$, we have
$\nabla r=\partial r$. It also satisfies $|\nabla r|=1$ where it is smooth.

Let $m(r)$ denote the mean curvature of the geodesic sphere at $p$ in the outer
normal direction. Then we have $m(r)=\Delta r$, where $\Delta$ is the Laplace
operator of $g$ (see \cite{[zhu]}). The classical mean curvature comparison
states that if
\[
\mathrm{Ric}\geq (n-1)H\,g,
\]
for some real constant $H$, then
\[
m(r)\leq m_H
\]
outside of the cut locus of $p$, where $m_H$ is the mean curvature of the geodesic
sphere in the model space $M_H^n$, the $n$-dimensional simply connected space with
constant sectional curvature $H$.

For any given smooth vector field
$V$ on $(M,g)$, we consider a diffusion operator instead of the usual Laplace operator
\[
\Delta_V:=\Delta-\langle V,\nabla\!\ \rangle,
\]
where $\langle\,,\,\rangle$ denotes the inner product with respect to the metric $g$.
Meanwhile the generalized mean curvature $m_V(r)$ associated to $V$ is defined by
\[
m_V(r):=m(r)-\langle V,\nabla r \rangle.
\]
Then we have $m_V(r)=\Delta_V(r)$.

\vspace{.1in}

In the following, we first give a rough estimate on $m_V$, requiring no assumptions
on $V$, which will be very useful in the proof of our main result. When we take
$V=\nabla f$ for some $f\in C^\infty(M)$, our result returns to the weighted
case considered by Wei and Wylie \cite{[WW]}.
\begin{theorem}\label{BasicMean}
Let $(M,g)$ be an $n$-dimensional complete Riemannian manifold. Assume that
$(M,g)$ admits a smooth vector field $V$ satisfying
\[
\mathrm{Ric}_V(\partial r,\partial r)\ge\lambda,\quad \lambda\in\mathbb{R}.
\]
Then given any minimal geodesic segment and $r_0 >0$,
\[
m_V(r)\le m_V(r_0)-\lambda (r-r_0) \ \  \mbox{for}\ r \ge r_0.
\]
Equality holds for some $r>r_0$ if and only if  all the radial sectional curvatures are zero,
$\mathrm{Hess}\, r\equiv 0$, and $\langle\nabla_{\partial r}V,\nabla r\rangle\equiv \lambda$ along
the geodesic from $r_0$ to $r$.
 \end{theorem}

\begin{proof}
Recall that the Bochner-Weitzenb\"ock formula
\[
\frac{1}{2}\Delta|\nabla u|^2
=|\mathrm{Hess}\,u|^2+\langle\nabla\Delta u,\nabla u\rangle
+{\rm Ric}(\nabla u,\nabla u)
\]
for any $u\in C^\infty(M)$. Letting $u(x)=r(x)$ and using $|\nabla r|=1$,
the above formula becomes
\begin{equation}\label{Bochdis}
0=|\mathrm{Hess}\,r|^2+\frac{\partial}{\partial r}\left(\Delta r\right)+{\rm Ric}(\nabla r,\nabla r).
\end{equation}
Here, $\mathrm{Hess}\,r$ is the second fundamental form of the geodesic sphere and
$\Delta r=m(r)$. We apply the Cauchy-Schwarz inequality to \eqref{Bochdis} and
obtain a Riccati inequality
\begin{equation}\label{Ricca}
m'\leq -\frac{m^2}{n-1}-{\rm Ric}(\partial r,\partial r).
\end{equation}
Notice that the Riccati inequality becomes equality if and only if the radial
sectional curvatures are constant. So the mean curvature $m_H$ of the model space
$M_H^n$ satisfies
\begin{equation}\label{inde}
m_H'=-\frac{m_H^2}{n-1}-(n-1)H.
\end{equation}
Noticing that
\begin{equation*}
\begin{split}
m_V'(r)&=m'(r)-\langle \nabla_{\partial r}V,\nabla r \rangle\\
&=m'(r)-\frac 12\mathcal{L}_V g(\partial r,\partial r),
\end{split}
\end{equation*}
from \eqref{Ricca} we have
\begin{equation}\label{recatti}
m'_V\leq -\frac{m^2}{n-1}-{\rm Ric}_V(\partial r,\partial r).
\end{equation}
By the assumption ${\rm Ric}_V(\partial r,\partial r)\geq \lambda$, then
\[
m'_V\leq-\lambda,
\]
which implies the inequality of Theorem \ref{BasicMean}.

Now we discuss the equality case. Assume that $m'_V\equiv-\lambda$
on an interval $[r_0, r]$. From \eqref{recatti} we have $m\equiv 0$  and
\[
m'_V=-\frac 12\mathcal{L}_V g(\partial r,\partial r)=-{\rm Ric}_V(\partial r,\partial r)=-\lambda.
\]
So we get ${\rm Ric}(\partial r,\partial r)=0$. Then by \eqref{Bochdis}, we further have
$\mathrm{Hess}\, r=0$, which implies the sectional curvatures must be zero.
\end{proof}

If the smooth vector field $V$ is bounded, we have a stronger comparison result, which
is a generalization of the Wei-Wylie comparison result (see Theorem 1.1 (a) in
\cite{[WW]}).
\begin{theorem}\label{meancomp}
Let $(M,g)$ be an $n$-dimensional complete Riemannian manifold. Assume that
$(M,g)$ admits a smooth vector field $V$ satisfying
\[
\mathrm{Ric}_V(\partial_r,\partial_r)\ge(n-1)H,\quad H\in\mathbb{R},
\]
along a minimal geodesic segment from a fixed point $p$ and $|V|\le a$
for some real constant $a\ge0$ (when $H>0$ assume $r\le\pi/2\sqrt{H}$). Then,
\[
m_V(r)-m_H(r)\le a
\]
along that minimal geodesic segment from $p$. Equality holds if and only if
the radial sectional curvatures are equal to $H$ and $V=-a\,\nabla r$.
\end{theorem}

\begin{proof}
Combining \eqref{Ricca} and  \eqref{inde}, and using the assumption on
$\mathrm{Ric}_V$, we have
\begin{equation}
\begin{split}\label{diffRic}
(m-m_H)'&\le-\frac{m^2-m^2_H}{n-1}+\frac 12\mathcal{L}_V g(\partial r,\partial r)\\
&=-\frac{m^2-m^2_H}{n-1}+\langle \nabla_{\partial r}V,\nabla r \rangle.
\end{split}
\end{equation}
To simply the above inequality, we introduce a new function $\mathrm{sn}_H(r)$, which
be the solution to
\[
\mathrm{sn}_H''+H\mathrm{sn}_H=0
\]
such that $\mathrm{sn}_H(0)=0$ and $\mathrm{sn}_H'(0)=1$. Note that
\[
m_H=(n-1)\frac{\mathrm{sn}'_H}{\mathrm{sn}_H}.
\]
Now using this function and the inequality \eqref{diffRic}, we compute that
\begin{equation*}
\begin{split}
\left[\mathrm{sn}^2_H(m-m_H)\right]'&=2\mathrm{sn}_H\mathrm{sn}'_H(m-m_H)+\mathrm{sn}^2_H(m-m_H)'\\
&\le\mathrm{sn}^2_H\left[\frac{2m_H}{n-1}(m-m_H)-\frac{m^2-m^2_H}{n-1}+\langle\nabla_{\partial r}V,\nabla r\rangle\right]\\
&=\mathrm{sn}^2_H\left[-\frac{(m-m_H)^2}{n-1}+\langle\nabla_{\partial r}V,\nabla r\rangle\right]\\
&\le\mathrm{sn}^2_H\langle\nabla_{\partial r}V,\nabla r\rangle.
\end{split}
\end{equation*}
Integrating the above inequality from $0$ to $r$, since $\mathrm{sn}_H(0)=0$, we have
\[
\mathrm{sn}^2_H(r)m(r)\le\mathrm{sn}^2_H(r)m_H(r)+\int^r_0\mathrm{sn}^2_H(t)\langle\nabla_{\partial t}V,\partial t\rangle dt.
\]
Noticing that integration by parts on the last term gives
\[
\int^r_0\mathrm{sn}^2_H(t)\langle\nabla_{\partial t}V,\partial t\rangle dt=\mathrm{sn}^2_H(r)\langle V,\nabla r\rangle -\int^r_0(\mathrm{sn}^2_H(t))'\langle V,\partial t\rangle dt,
\]
we duduce
\begin{equation}\label{simp}
\mathrm{sn}^2_H(r)m_V(r)\le\mathrm{sn}^2_H(r)m_H(r)-\int^r_0(\mathrm{sn}^2_H(t))'\langle V,\partial t\rangle dt.
\end{equation}
By our theorem assumptions we know that
\[
(\mathrm{sn}^2_H(t))'=2\mathrm{sn}_H(t)\mathrm{sn}'_H(t)\ge 0.
\]
Therefore, if $|V|\leq a$, then
\[
\mathrm{sn}^2_H(r)m_V(r)\le\mathrm{sn}^2_H(r)m_H(r)+a\int^r_0(\mathrm{sn}^2_H(t))'dt.
\]
That is,
\[
\mathrm{sn}^2_H(r)m_V(r)\le\mathrm{sn}^2_H(r)m_H(r)+a\,\mathrm{sn}^2_H(r),
\]
which proves the theorem.

To discuss the equality case, suppose that $m_V(r)=m_H(r)+a$ for some $r$ and $|V|\leq a$.
By \eqref{simp}, we have
\[
a\,\mathrm{sn}^2_H(r)\le -\int^r_0(\mathrm{sn}^2_H(t))'\langle V,\partial t\rangle dt\le a\,\mathrm{sn}^2_H(r).
\]
So $V=-a \nabla r$. Therefore,
\[
m(r)=m_V(r)+\langle V,\nabla r \rangle=m_H(r)+a+\langle-a \nabla r,\nabla r \rangle=m_H(r),
\]
which means the rigidity follows from the rigidity for the usual mean curvature comparison.
\end{proof}

\section{Proof of Theorems \ref{Thm} and \ref{Thm2}}\label{sec3}
In this section we will prove Theorems \ref{Thm} and \ref{Thm2}. The proof uses
the generalized mean curvature comparison applied to the excess function. The proof trick
was also used by Wei and Wylie \cite{[WW]} to prove the Myers' type theorem on
smooth metric measure spaces $(M, g,e^{-f}dv)$ when $f$ is bounded.

\begin{proof}[Proof of Theorem \ref{Thm}]
Let $(M,g)$ admits a smooth vector field $V$ such that
\[
\mathrm{Ric}_V \geq (n-1)H\,g,
\]
for some real constant $H>0$, and $|V|\leq a$ for some real constant $a\ge0$.
Let $p_1,p_2$ are two points in $M$ with $d(p_1, p_2) \geq \frac{\pi}{ \sqrt{H}}$ and
set
\[
B:=d(p_1,p_2)-\frac{\pi}{\sqrt{H}}.
\]
Let $r_1(x)=d(p_1,x)$ and $r_2(x)=d(p_2,x)$. Denote $e(x)$ by the excess function for
the points $p_1$ and $p_2$, i.e.
\[
e(x):=d(p_1,x)+d(p_2,x)-d(p_1,p_2),
\]
which measures how much the triangle inequality fails to be an equality. By the
triangle inequality, we have
\[
e(x)\geq 0\quad \mathrm{and}\quad e(\gamma(t))=0,
\]
where $\gamma$ is a minimal geodesic from $p_1$ to $p_2$. Hence
\[
\Delta_V(e) (\gamma(t)) \geq 0
\]
in the barrier sense.
Let
\[
y_1=\gamma\left(\frac{\pi}{2\sqrt{H}}\right)
\quad\mathrm{and}\quad
y_2=\gamma\left(\frac{\pi}{2\sqrt{H}}+B\right).
\]
Then it is easy to see that $r_i(y_i)= \frac{\pi}{2\sqrt{H}}$, $i=1$ and $2$.
Furthermore, by Theorem \ref{meancomp}, we have
\begin{equation}\label{meanest1}
\Delta_V(r_i(y_i))\leq a,\quad i=1\,\,\mathrm{and}\,\,2.
\end{equation}
We would like to point out that we can not give a estimate for $\Delta_V(r_1(y_2))$
by directly using Theorem \ref{meancomp}, since $r_1(y_2)>\frac{\pi}{2\sqrt{H}}$.
But we can apply Theorem \ref{BasicMean} and estimate \eqref{meanest1} to estimate
\begin{equation}\label{meanest2}
\Delta_V(r_1(y_2))\le a-B(n-1)H.
\end{equation}
Therefore,
\begin{equation*}
\begin{split}
0\leq\Delta_V(e)(y_2)
&=\Delta_V(r_1)(y_2)+\Delta_V(r_2)(y_2)\\
&\le 2a-B(n-1)H,
\end{split}
\end{equation*}
where we used estimates \eqref{meanest1} and \eqref{meanest2}. Hence we have
\[
B\le\frac{2a}{(n-1)\sqrt{H}}
\]
and
\[
d(p_1, p_2)\le\frac{\pi}{\sqrt{H}}+\frac{2a}{(n-1)H}.
\]
Since $p_1$ and $p_2$ are arbitrary two points, this completes the proof.
\end{proof}

\vspace{.1in}

In the rest of this section, we now give a short explanation on how to prove
Theorem \ref{Thm2}. Indeed, its proof is almost the same as the case of
Theorem \ref{Thm}: the difference is that we use Wei-Wylie's mean curvature
comparisons (Theorems 1.1 (a) and 3.1 in \cite{[WW]}) instead of Theorems
\ref{BasicMean} and \ref{meancomp}. So we omit it here.

\section{Proof of Corollary \ref{cor}}\label{sec4}
We first recall a Fern\'andez-L\'opez and Garc\'ia-R\'io's result \cite{[FG2]},
which states that the gradient norm of potential function of compact shrinking
Ricci solitons can be controlled by the scalar curvature.

\begin{lemma}\label{lem}
Let $(M,g)$ be an $n$-dimensional compact Riemannian manifold which
admits a smooth function $f$ satisfying
\[
\mathrm{Ric}+\mathrm{Hess}\, f=\lambda\, g
\]
for some real constant $\lambda>0$. Then
\[
|\nabla f|\, \leq\sqrt{R_{max}-R}.
\]
\end{lemma}

Using Lemma \ref{lem}, we can prove Corollary \ref{cor}.
\begin{proof}[Proof of Corollary \ref{cor}]
In Theorem \ref{Thm}, let $V=\nabla f$, and then
\[
a=\sqrt{R_{max}-R_{min}}.
\]
Meanwhile we let $(n-1)H=\lambda$. Substituting these into \eqref{newest}
proves the estimate.
\end{proof}


\bibliographystyle{amsplain}

\end{document}